\documentclass{article}%
\usepackage[colorlinks=true, pdfstartview=FitV, linkcolor=blue,
citecolor=blue, urlcolor=blue]{hyperref}
\usepackage{graphicx,color}
\usepackage{amssymb}
\usepackage{amsfonts}
\usepackage{amsmath}
\usepackage{graphicx,color}
\usepackage{graphicx}%
\providecommand{\U}[1]{\protect\rule{.1in}{.1in}}
\newtheorem{theorem}{Theorem}

\newtheorem{corollary}[theorem]{Corollary}

\newtheorem{example}[theorem]{Example}

\newtheorem{lemma}[theorem]{Lemma}

\newtheorem{remark}[theorem]{Remark}

\newenvironment{proof}[1][Proof]{\noindent\textbf{#1.} }{\ \rule{0.5em}{0.5em}}
\begin{document}

\title{Extended Bernoulli and Stirling matrices and related combinatorial identities
\thanks{Accepted for publication in \textit{Linear Algebra and its
Applications}. DOI: 10.1016/j.laa.2013.11.031}}
\author{M\"{u}m\"{u}n Can\thanks{mcan@akdeniz.edu.tr} and M. Cihat Da\u{g}l\i
\ \thanks{mcihatdagli@akdeniz.edu.tr}\\Department of Mathematics, Akdeniz University, 07058-Antalya, Turkey}
\date{}
\maketitle

\begin{abstract}
In this paper we establish plenty of number theoretic and combinatoric
identities involving generalized Bernoulli and Stirling numbers of both kinds.
These formulas are deduced from Pascal type matrix representations of
Bernoulli and Stirling numbers. For this we define and factorize a modified
Pascal matrix corresponding to Bernoulli and Stirling cases. \bigskip

\bigskip


\end{abstract}

\section{Introduction}

Matrices and matrix theory are recently used in number theory and
combinatorics. In particular Pascal type lower-triangular matrices are studied
with Fibonacci, Bernoulli, Stirling and Pell numbers and other special numbers
sequences. Cheon and Kim \cite{6} factorized (generalized) Stirling matrices
by Pascal matrices and obtained some combinatorial identities.\ Zhang and Wang
\cite{7} gave product formulas for the Bernoulli matrix and established
several identities involving Fibonacci numbers, Bernoulli numbers and polynomials.

In this paper we employ matrices for degenerate Bernoulli polynomials and
generalized Stirling numbers. We define degenerate Bernoulli and generalized
Stirling matrices which generalize previous results and lead some new
combinatorial identities. Some of these identities can hardly be obtained by
classical ways such as by using generating functions or counting, however they
are easily come up via matrix representations after elementary matrix multiplication.

The summary by sections is as follows: In Section 2, we define Pascal
functional matrix which is a special case of Pascal functional matrices
defined in \cite{16,20} and factorize by the summation matrices. In Section 3,
we generalize Bernoulli matrix and investigate some properties. In Section 4,
we define two types generalized Stirling matrices and obtain relationships
between\textbf{\ }Bernoulli matrices and Stirling matrices of the second type.
Furthermore, degenerate Bernoulli and generalized Stirling matrices are
factorized by Pascal matrices and several identities are developed as a result
of matrix representations. In final section, we introduce some special cases
of the results obtained in Section 4.

Throughout this paper we assume that $i,$\ $j$\ and $n$\ are nonnegative
integers; $\mu,$\ $\lambda,$\ $w$\ and $x$\ are real or complex numbers.

\section{Pascal matrix}

We define the $n\times n$ Pascal functional matrix $\mathcal{P}_{n}\left[
\lambda,x\right]  $ by%
\[
\left(  \mathcal{P}_{n}\left[  \lambda,x\right]  \right)  _{i,j}=\left\{
\begin{array}
[c]{cl}%
\dbinom{i-1}{j-1}\left(  x|\lambda\right)  _{i-j} & \text{, if }i\geq j\geq1\\
0 & \text{, if }1\leq i<j,
\end{array}
\right.
\]
where $\left(  x|\lambda\right)  _{k}=x(x-\lambda)(x-2\lambda)\cdots
(x-(k-1)\lambda)$ with $\left(  x|\lambda\right)  _{0}=1.$ Since
\[
\left(  x+y|\lambda\right)  _{m}=\sum\limits_{k=0}^{m}\dbinom{m}{k}\left(
x|\lambda\right)  _{m-k}\left(  y|\lambda\right)  _{k},
\]
we have $\mathcal{P}_{n}\left[  \lambda,x+y\right]  =\mathcal{P}_{n}\left[
\lambda,x\right]  \mathcal{P}_{n}\left[  \lambda,y\right]  $ and
$\mathcal{P}_{n}^{-1}\left[  \lambda,x\right]  =\mathcal{P}_{n}\left[
\lambda,-x\right]  . $ The algebraic properties of Pascal matrices can be
found in \cite{5,21,9,16,2,22,20}. In fact, $\mathcal{P}_{n}\left[
-\lambda,x\right]  $ is the matrix $\mathcal{P}_{n,\lambda}\left[  x\right]  $
defined in \cite{5} and this matrix is a special case of the generalized
Pascal functional matrices defined in \cite{16,20}. So we will not discuss the
algebraic properties of this matrix. We will only focus on factorizing this
matrix by the summation matrices. For this\textbf{\ }purpose, let us define
the $n\times n$ matrices $\mathcal{R}_{n}\left[  \lambda,x\right]
\mathcal{=}\left[  R_{i,j}(\lambda,x)\right]  $ and $\mathcal{T}_{n}\left[
\lambda,x\right]  \mathcal{=}\left[  T_{i,j}(\lambda,x)\right]  $
$(i,j=1,2,...,n)$ by%
\[
R_{i,j}\left(  \lambda,x\right)  =\left\{
\begin{array}
[c]{cc}%
\dfrac{\left(  x|\lambda\right)  _{i-1}}{\left(  x-\lambda|\lambda\right)
_{j-1}} & \text{, if }i>j\\
1 & \text{, if }i=j\\
0 & \text{, if }i<j
\end{array}
\right.
\]
and%
\[
T_{i,j}\left(  \lambda,x\right)  =\left\{
\begin{array}
[c]{cc}%
(-1)^{i-j}\lambda^{i-j-1}\dfrac{(i-2)!}{\left(  j-1\right)  !}x & \text{, if
}i>j\\
1 & \text{, if }i=j\\
0 & \text{, if }i<j.
\end{array}
\right.
\]
Then, we have

\begin{theorem}%
\[
\mathcal{R}_{n}^{-1}\left[  \lambda,x\right]  \mathcal{=T}_{n}\left[
\lambda,x\right]  .
\]

\end{theorem}

\begin{proof}
We apply induction on $i$ to show that
\[
\left(  \mathcal{R}_{n}\left[  \lambda,x\right]  \mathcal{T}_{n}\left[
\lambda,x\right]  \right)  _{i,j}=\sum\limits_{k=j}^{i}R_{i,k}\left(
\lambda,x\right)  T_{k,j}\left(  \lambda,x\right)  =0
\]
when $i>j$. For $i=j+1$, the assertion is trivial. Suppose that the equation
is true for $i=m\geq j+2$, i.e.,
\begin{align*}
&  \frac{\left(  x|\lambda\right)  _{m-1}}{\left(  x-\lambda|\lambda\right)
_{j-1}}+\sum\limits_{k=j+1}^{m-1}\frac{\left(  x|\lambda\right)  _{m-1}%
}{\left(  x-\lambda|\lambda\right)  _{k-1}}\frac{(k-2)!}{\left(  j-1\right)
!}\left(  -1\right)  ^{k-j}\lambda^{k-j-1}x\\
&  \qquad\qquad\qquad\qquad\qquad\qquad\qquad+\frac{(m-2)!}{\left(
j-1\right)  !}\left(  -1\right)  ^{m-j}\lambda^{m-j-1}x=0.
\end{align*}
For $i=m+1,$ we have%
\begin{align*}
&  \frac{\left(  x|\lambda\right)  _{m}}{\left(  x-\lambda|\lambda\right)
_{j-1}}+\sum\limits_{k=j+1}^{m}\frac{\left(  x|\lambda\right)  _{m}}{\left(
x-\lambda|\lambda\right)  _{k-1}}\frac{(k-2)!}{\left(  j-1\right)
!}(-1)^{k-j}\lambda^{k-j-1}x\\
&  \qquad+\frac{(m-1)!}{\left(  j-1\right)  !}(-1)^{m+1-j}\lambda^{m-j}x\\
&  \quad=\frac{\left(  x|\lambda\right)  _{m}}{\left(  x-\lambda
|\lambda\right)  _{j-1}}+\sum\limits_{k=j+1}^{m-1}\frac{\left(  x|\lambda
\right)  _{m}}{\left(  x-\lambda|\lambda\right)  _{k-1}}\frac{(k-2)!}{\left(
j-1\right)  !}(-1)^{k-j}\lambda^{k-j-1}x\\
&  \qquad-\frac{\left(  x|\lambda\right)  _{m}}{\left(  x-\lambda
|\lambda\right)  _{m-1}}\frac{(m-2)!}{\left(  j-1\right)  !}\left(
-\lambda\right)  ^{m-j-1}x-\frac{(m-1)!}{\left(  j-1\right)  !}\left(
-\lambda\right)  ^{m-j}x
\end{align*}%
\begin{align*}
&  \quad=(x-\left(  m-1\right)  \lambda)\left\{  \sum\limits_{k=j+1}%
^{m-1}\frac{\left(  x|\lambda\right)  _{m-1}}{\left(  x-\lambda|\lambda
\right)  _{k-1}}\frac{(k-2)!}{\left(  j-1\right)  !}(-1)^{k-j}\lambda
^{k-j-1}x\right. \\
&  \qquad\qquad\qquad\qquad\qquad\left.  +\frac{\left(  x|\lambda\right)
_{m-1}}{\left(  x-\lambda|\lambda\right)  _{j-1}}+(-1)^{m-j}\frac
{(m-2)!}{\left(  j-1\right)  !}\lambda^{m-1-j}x\right\}  =0.
\end{align*}
Since
\[
\left(  \mathcal{R}_{n}\left[  \lambda,x\right]  \mathcal{T}_{n}\left[
\lambda,x\right]  \right)  _{i,j}=0
\]
for $i<j,$ and
\[
\left(  \mathcal{R}_{n}\left[  \lambda,x\right]  \mathcal{T}_{n}\left[
\lambda,x\right]  \right)  _{i,i}=R_{i,i}\left(  \lambda,x\right)
T_{i,i}\left(  \lambda,x\right)  =1,
\]
we get $\left(  \mathcal{R}_{n}\left[  \lambda,x\right]  \mathcal{T}%
_{n}\left[  \lambda,x\right]  \right)  _{i,j}=\delta_{i,j},$ where
$\delta_{i,j}$ is the Kronecker symbol.
\end{proof}

Furthermore, we need the $\left(  k+1\right)  \times\left(  k+1\right)  $
matrices
\[
\overline{\mathcal{P}}_{k}\left[  \lambda,x\right]  =\left[  1\right]
\oplus\mathcal{P}_{k}\left[  \lambda,x\right]  ,\text{ }k\geq1
\]
and the $n\times n$ matrices
\[
G_{k}\left[  \lambda,x\right]  =I_{n-k}\oplus\mathcal{R}_{k}\left[
\lambda,x\right]  ,\text{ }1\leq k\leq n-1\text{ and }G_{n}\left[
\lambda,x\right]  =\mathcal{R}_{n}\left[  \lambda,x\right]  ,
\]
where the notation $\oplus$ denotes the direct sum of two matrices and $I_{n}$
is the identity matrix of order $n$.

\begin{lemma}
\label{lem rp}For $k\geq1,$ we have%
\[
\mathcal{R}_{k}\left[  \lambda,x\right]  \overline{\mathcal{P}}_{k-1}\left[
\lambda,x\right]  =\mathcal{P}_{k}\left[  \lambda,x\right]  .
\]

\end{lemma}

\begin{proof}
We must show that
\begin{equation}
\sum\limits_{r=j}^{i-1}\frac{\left(  x|\lambda\right)  _{i-1}}{\left(
x-\lambda|\lambda\right)  _{r-1}}\dbinom{r-2}{j-2}\left(  x|\lambda\right)
_{r-j}+\dbinom{i-2}{j-2}\left(  x|\lambda\right)  _{i-j}=\dbinom{i-1}%
{j-1}\left(  x|\lambda\right)  _{i-j}, \label{15}%
\end{equation}
for $i\geq j$, since the left-hand side of (\ref{15}) is the $(i,j)$-entry of
the matrix $\mathcal{R}_{k}\left[  \lambda,x\right]  \overline{\mathcal{P}%
}_{k-1}\left[  \lambda,x\right]  .$ Again, we apply induction on $i.$ For
$i=j$ , the assertion is clear. From the known property $\dbinom{i-1}%
{j-1}+\dbinom{i-1}{j}=\dbinom{i}{j},$ it is enough to show that
\begin{equation}
\left(  x|\lambda\right)  _{i-1}\sum\limits_{r=j}^{i-1}\dbinom{r-2}{j-2}%
\frac{\left(  x|\lambda\right)  _{r-j}}{\left(  x-\lambda|\lambda\right)
_{r-1}}=\dbinom{i-2}{j-1}\left(  x|\lambda\right)  _{i-j} \label{16}%
\end{equation}
\ for $i\geq j+1.$ Suppose that (\ref{16}) is true for $i=m>j.$ For $i=m+1,$
we have
\begin{align*}
&  \left(  x|\lambda\right)  _{m}\sum\limits_{r=j}^{m}\dbinom{r-2}{j-2}%
\frac{\left(  x|\lambda\right)  _{r-j}}{\left(  x-\lambda|\lambda\right)
_{r-1}}\\
&  \quad=\left(  x|\lambda\right)  _{m}\sum\limits_{r=j}^{m-1}\dbinom
{r-2}{j-2}\frac{\left(  x|\lambda\right)  _{r-j}}{\left(  x-\lambda
|\lambda\right)  _{r-1}}+\dbinom{m-2}{j-2}\frac{\left(  x|\lambda\right)
_{m}}{\left(  x-\lambda|\lambda\right)  _{m-1}}\left(  x|\lambda\right)
_{m-j}\\
&  \quad=\left[  x-\left(  m-1\right)  \lambda\right]  \dbinom{m-2}%
{j-1}\left(  x|\lambda\right)  _{m-j}+x\dbinom{m-2}{j-2}\left(  x|\lambda
\right)  _{m-j}\\
&  \quad=\dbinom{m-1}{j-1}\left[  x-(m-j)\lambda\right]  \left(
x|\lambda\right)  _{m-j}=\dbinom{m-1}{j-1}\left(  x|\lambda\right)  _{m+1-j}.
\end{align*}
This completes the proof.
\end{proof}

From the definition of the matrices $G_{k}\left[  \lambda,x\right]  $ and
Lemma \ref{lem rp}, we have the following factorization of $\mathcal{P}%
_{n}\left[  \lambda,x\right]  ,$ which generalizes the result of Zhang
\cite[Theorem 1]{2}.

\begin{theorem}
\label{teo pg}%
\[
\mathcal{P}_{n}\left[  \lambda,x\right]  =G_{n}\left[  \lambda,x\right]
G_{n-1}\left[  \lambda,x\right]  \cdots G_{1}\left[  \lambda,x\right]  .
\]

\end{theorem}

\begin{example}
$G_{4}\left[  \lambda,x\right]  G_{3}\left[  \lambda,x\right]  G_{2}\left[
\lambda,x\right]  G_{1}\left[  \lambda,x\right]  \medskip$

$\hspace{-0.2in}=%
\begin{bmatrix}
1 & 0 & 0 & 0\\
x & 1 & 0 & 0\\
x^{2}-x\lambda & x & 1 & 0\\
x^{3}-3x^{2}\lambda+2x\lambda^{2} & x^{2}-2x\lambda & x & 1
\end{bmatrix}%
\begin{bmatrix}
1 & 0 & 0 & 0\\
0 & 1 & 0 & 0\\
0 & x & 1 & 0\\
0 & x^{2}-x\lambda & x & 1
\end{bmatrix}%
\begin{bmatrix}
1 & 0 & 0 & 0\\
0 & 1 & 0 & 0\\
0 & 0 & 1 & 0\\
0 & 0 & x & 1
\end{bmatrix}
\medskip$

$\hspace{-0.2in}=%
\begin{bmatrix}
1 & 0 & 0 & 0\\
x & 1 & 0 & 0\\
x^{2}-x\lambda & 2x & 1 & 0\\
x^{3}-3x^{2}\lambda+2x\lambda^{2} & 3x^{2}-3x\lambda & 3x & 1
\end{bmatrix}
=\mathcal{P}_{4}\left[  \lambda,x\right]  . \medskip$
\end{example}

\section{Degenerate Bernoulli matrices}

\subsection{Degenerate Bernoulli polynomials of the first kind}

The higher order degenerate Bernoulli polynomials of the first kind $\beta
_{m}^{(w)}(\lambda,x)$ are defined by means of the generating function
\cite{8}
\begin{equation}
\left(  \frac{t}{(1+\lambda t)^{1/\lambda}-1}\right)  ^{w}(1+\lambda
t)^{x/\lambda}=\sum_{m=0}^{\infty}\beta_{m}^{(w)}(\lambda,x)\frac{t^{m}}{m!}
\label{1}%
\end{equation}
for $\lambda\neq0.$ Clearly, $\beta_{m}^{(1)}(\lambda,x)=\beta_{m}(\lambda,x)$
and $\beta_{m}^{(1)}(\lambda,0)=\beta_{m}(\lambda)$ are the degenerate
Bernoulli polynomials and the degenerate Bernoulli numbers, respectively. The
first few of the degenerate Bernoulli polynomials are $\beta_{0}%
(\lambda,x)=1,$ $\beta_{1}(\lambda,x)=x+\frac{1}{2}\lambda-\frac{1}{2}%
,$\ $\beta_{2}(\lambda,x)=x^{2}-x-\frac{1}{6}\lambda^{2}+\frac{1}{6},$%
\ $\beta_{3}(\lambda,x)=x^{3}-\frac{3}{2}x^{2}+\frac{1}{2}x-\frac{3}{2}%
x^{2}\lambda+\frac{3}{2}x\lambda+\frac{1}{4}\lambda^{3}-\frac{1}{4}\lambda,$
$\beta_{4}(\lambda,x)=x^{4}-2x^{3}+x^{2}-4x^{3}\lambda+4x^{2}\lambda
^{2}+6x^{2}\lambda-4x\lambda^{2}-2x\lambda-\frac{19}{30}\lambda^{4}+\frac
{2}{3}\lambda^{2}-\frac{1}{30}.$ From (\ref{1}), we have
\begin{equation}
\beta_{m}^{(w+z)}(\lambda,x+y)=\sum\limits_{k=0}^{m}\dbinom{m}{k}\beta
_{k}^{(w)}(\lambda,x)\beta_{m-k}^{(z)}(\lambda,y) \label{0}%
\end{equation}
and
\[
\beta_{m}\left(  \lambda,x+y\right)  =\sum\limits_{k=0}^{m}\dbinom{m}{k}%
\beta_{k}(\lambda,x)\left(  y|\lambda\right)  _{m-k}%
\]
(cf. \cite[eq. (5.12)]{8}). Explicit formulas and recurrence relations of
(generalized) degenerate Bernoulli polynomials and numbers can be found in
\cite{1,4,8,3,13,15}. Divisibility properties (\cite{4,13,10,15}\textbf{) }and
symmetry relations (\cite{3a,12,15}) are demonstrated as well.

Let $\mathcal{B}_{n}^{(w)}\left[  \lambda,x\right]  $ be the $n\times n$
matrix defined by
\[
\left(  \mathcal{B}_{n}^{(w)}\left[  \lambda,x\right]  \right)  _{i,j}%
=\left\{
\begin{array}
[c]{cl}%
\dbinom{i-1}{j-1}\beta_{i-j}^{(w)}(\lambda,x) & \text{, if }i\geq j\geq1\\
0 & \text{, if }1\leq i<j,
\end{array}
\right.
\]
with the notations $\mathcal{B}_{n}^{(1)}\left[  \lambda,x\right]
=\mathcal{B}_{n}\left[  \lambda,x\right]  $ and $\mathcal{B}_{n}^{(1)}\left[
\lambda,0\right]  =\mathcal{B}_{n}\left[  \lambda\right]  $. It is clear that
$\mathcal{B}_{n}^{(0)}\left[  \lambda,x\right]  =\mathcal{P}_{n}\left[
\lambda,x\right]  .$

Since $\left(  1+\lambda t\right)  ^{1/\lambda}\rightarrow e^{t},$ as
$\lambda\rightarrow0$ it is evident that $\beta_{m}^{(w)}(0,x)=B_{m}^{(w)}(x)$
and $\beta_{m}^{\left(  w\right)  }(0,0)=B_{m}^{\left(  w\right)  },$ where
$B_{m}^{(w)}(x)$ are the higher order Bernoulli polynomials defined by
\[
\left(  \frac{t}{e^{t}-1}\right)  ^{w}e^{xt}=\sum_{m=0}^{\infty}B_{m}%
^{(w)}(x)\frac{t^{m}}{m!}.
\]
Hence, in the limiting case $\lambda=0,$ $\mathcal{B}_{n}^{(w)}\left[
0,x\right]  $ is the generalized Bernoulli matrix $\mathcal{B}_{n-1}%
^{(w)}\left[  x\right]  $ defined in \cite[p. 1623]{7}.

The following theorem can be easily seen from (\ref{0}).

\begin{theorem}%
\[
\mathcal{B}_{n}^{(w+z)}\left[  \lambda,x+y\right]  =\mathcal{B}_{n}%
^{(w)}\left[  \lambda,x\right]  \mathcal{B}_{n}^{(z)}\left[  \lambda,y\right]
=\mathcal{B}_{n}^{(z)}\left[  \lambda,x\right]  \mathcal{B}_{n}^{(w)}\left[
\lambda,y\right]  .
\]

\end{theorem}

By induction on $k,$ we have

\begin{corollary}
\label{cor bk}%
\[
\mathcal{B}_{n}^{(w_{1}+w_{2}+\cdots+w_{k})}\left[  \lambda,x_{1}+x_{2}%
+\cdots+x_{k}\right]  =\mathcal{B}_{n}^{(w_{1})}\left[  \lambda,x_{1}\right]
\mathcal{B}_{n}^{(w_{2})}\left[  \lambda,x_{2}\right]  \cdots\mathcal{B}%
_{n}^{(w_{k})}\left[  \lambda,x_{k}\right]  .
\]
In particular,%
\[
\left(  \mathcal{B}_{n}^{(w)}\left[  \lambda,x\right]  \right)  ^{k}%
=\mathcal{B}_{n}^{(kw)}\left[  \lambda,kx\right]  .
\]

\end{corollary}

In view of $\mathcal{B}_{n}^{(0)}\left[  \lambda,0\right]  =I_{n}$ and
$\mathcal{B}_{n}^{(0)}\left[  \lambda,x\right]  =\mathcal{P}_{n}\left[
\lambda,x\right]  $ we have

\begin{corollary}
\label{teo bp}%
\begin{equation}
\mathcal{B}_{n}^{\left(  w\right)  }\left[  \lambda,x+y\right]  =\mathcal{P}%
_{n}\left[  \lambda,x\right]  \mathcal{B}_{n}^{\left(  w\right)  }\left[
\lambda,y\right]  =\mathcal{B}_{n}^{\left(  w\right)  }\left[  \lambda
,y\right]  \mathcal{P}_{n}\left[  \lambda,x\right]  \label{bp}%
\end{equation}
and%
\[
\left(  \mathcal{B}_{n}^{(w)}\left[  \lambda,x\right]  \right)  ^{-1}%
=\mathcal{B}_{n}^{(-w)}\left[  \lambda,-x\right]  =\mathcal{P}_{n}\left[
\lambda,-x\right]  \mathcal{B}_{n}^{\left(  -w\right)  }\left[  \lambda
\right]  .
\]

\end{corollary}

The following is a consequence of Theorem \ref{teo pg} and Corollary
\ref{teo bp}.

\begin{corollary}%
\[
\mathcal{B}_{n}\left[  \lambda,x\right]  =G_{n}\left[  \lambda,x\right]
G_{n-1}\left[  \lambda,x\right]  \cdots G_{1}\left[  \lambda,x\right]
\mathcal{B}_{n}\left[  \lambda\right]  .
\]

\end{corollary}

Consider the matrix\textbf{\ }%
\begin{align*}
\left(  \mathcal{B}_{n}\left[  \lambda,x\right]  -I_{n}\right)  ^{h}  &
=\sum\limits_{k=0}^{h}\binom{h}{k}\left(  -1\right)  ^{h-k}\left(
\mathcal{B}_{n}\left[  \lambda,x\right]  \right)  ^{k}\\
&  =\sum\limits_{k=0}^{h}\binom{h}{k}\left(  -1\right)  ^{h-k}\mathcal{B}%
_{n}^{\left(  k\right)  }\left[  \lambda,kx\right]
\end{align*}
for positive integer $h$. Since $diag\left(  \mathcal{B}_{n}\left[
\lambda,x\right]  -I_{n}\right)  =\left(  0,0,\cdots,0\right)  $ and $\left(
\mathcal{B}_{n}\left[  \lambda,x\right]  -I_{n}\right)  $ is a
lower-triangular matrix, it follows that $\left(  \mathcal{B}_{n}\left[
\lambda,x\right]  -I_{n}\right)  ^{h}=\left[  0\right]  _{n\times n}$ for
$n\leq h.$ Then%
\[
\left(  \mathcal{B}_{n}\left[  \lambda,\frac{x}{h}\right]  \right)
^{h}=\mathcal{B}_{n}^{\left(  h\right)  }\left[  \lambda,x\right]
=\sum\limits_{k=0}^{h-1}\binom{h}{k}\left(  -1\right)  ^{h-1-k}\mathcal{B}%
_{n}^{\left(  k\right)  }\left[  \lambda,\frac{k}{h}x\right]  ,\text{ for
}1\leq n\leq h.
\]
This yields%
\[
\sum\limits_{k=0}^{h-1}\binom{h}{k}\left(  -1\right)  ^{h-1-k}\beta
_{m}^{\left(  k\right)  }\left(  \lambda,\frac{k}{h}x\right)  =\beta
_{m}^{\left(  h\right)  }\left(  \lambda,x\right)  ,\text{ for }0\leq m<h.
\]
By the known identity $\beta_{m}^{\left(  h\right)  }\left(  \lambda,1\right)
=m\beta_{m-1}^{\left(  h-1\right)  }\left(  \lambda\right)  +\beta
_{m}^{\left(  h\right)  }\left(  \lambda\right)  $ for $m\geq1,$ we have
\[
\sum\limits_{k=0}^{h}\binom{h}{k}\left(  -1\right)  ^{h-k}\beta_{m}^{\left(
k\right)  }\left(  \lambda,\frac{k}{h}\right)  =-m\beta_{m-1}^{\left(
h-1\right)  }\left(  \lambda\right)  ,\text{ for }1\leq m<h.
\]
Similarly, we may get
\[
\sum\limits_{k=0}^{h-1}\binom{h}{k}\left(  -1\right)  ^{h-1-k}\left(
kx|\lambda\right)  _{m}=\left(  hx|\lambda\right)  _{m},\text{ for }0\leq
m<h,
\]
by the\emph{\ }fact that $\left(  \mathcal{P}_{n}\left[  \lambda,x\right]
\right)  ^{h}=\mathcal{P}_{n}\left[  \lambda,hx\right]  .$

\subsection{Degenerate Bernoulli polynomials of the second kind}

The higher order degenerate Bernoulli polynomials of the second kind
$\alpha_{m}^{(w)}(\lambda,x)$ are defined by \cite{1}
\begin{equation}
\left(  \frac{\lambda t}{(1+t)^{\lambda}-1}\right)  ^{w}(1+t)^{x}=\sum
_{m=0}^{\infty}\alpha_{m}^{(w)}(\lambda,x)\frac{t^{m}}{m!}. \label{1a}%
\end{equation}
For $x=0$, $\alpha_{m}^{(w)}(\lambda,0)=\alpha_{m}^{(w)}(\lambda)$ are called
the higher order degenerate Bernoulli numbers of the second kind. In the
limiting case $\lambda=0$ we have $\alpha_{m}^{(w)}\left(  0,x\right)
=m!b_{m}^{\left(  w\right)  }\left(  x\right)  ,$\ where $b_{m}^{\left(
w\right)  }\left(  x\right)  $\ are the higher order Bernoulli polynomials of
the second kind defined by
\[
\left(  \frac{t}{\log(1+t)}\right)  ^{w}(1+t)^{x}=\sum_{m=0}^{\infty}%
b_{m}^{(w)}\left(  x\right)  t^{m}.
\]
It is clear from (\ref{1}) and (\ref{1a}) that
\begin{equation}
\beta_{m}^{(w)}\left(  \frac{1}{\lambda},\frac{x}{\lambda}\right)  =\left(
\frac{1}{\lambda}\right)  ^{m}\alpha_{m}^{(w)}\left(  \lambda,x\right)  .
\label{1b2}%
\end{equation}

Let $\mathcal{L}_{n}^{\left(  w\right)  }\left[  \lambda,x\right]  $ be the
$n\times n$ matrix defined by
\[
\left(  \mathcal{L}_{n}^{\left(  w\right)  }\left[  \lambda,x\right]  \right)
_{i,j}=\left\{
\begin{array}
[c]{ll}%
\binom{i-1}{j-1}\alpha_{i-j}^{\left(  w\right)  }\left(  \lambda,x\right)  &
\text{, if }i\geq j\geq1,\\
0 & \text{, if }i<j.
\end{array}
\right.
\]
It is obvious that $\mathcal{L}_{n}^{(0)}\left[  \lambda,x\right]
=\mathcal{P}_{n}\left[  1,x\right]  .$ One can observe that $\mathcal{L}%
_{n}^{\left(  w\right)  }\left[  \lambda,x\right]  $\ satisfy properties given
for $\mathcal{B}_{n}^{\left(  w\right)  }\left[  \lambda,x\right]  ,$ however,
we prefer not to list them here.

\section{Generalized Stirling matrices}

Recall the generalized Stirling numbers of the first and of the second kinds.
For nonnegative integer $m$ and real or complex parameters $\mu,$ $\lambda$
and $x,$ with $\left(  \mu,\lambda,x\right)  \not =\left(  0,0,0\right)  ,$
the generalized Stirling numbers of the first kind $S_{1}\left(
m,k|\mu,\lambda,x\right)  $ and of the second kind $S_{2}\left(
m,k|\mu,\lambda,x\right)  $\ are defined by means of the generating functions
(cf. \cite[p. 372]{17})
\begin{align}
\left(  \frac{\left(  1+\mu t\right)  ^{\lambda/\mu}-1}{\lambda}\right)
^{k}\left(  1+\mu t\right)  ^{x/\mu}  &  =k!\sum\limits_{m=0}^{\infty}%
S_{1}\left(  m,k|\mu,\lambda,x\right)  \frac{t^{m}}{m!},\label{stx1}\\
\left(  \frac{\left(  1+\lambda t\right)  ^{\mu/\lambda}-1}{\mu}\right)
^{k}\left(  1+\lambda t\right)  ^{-x/\lambda}  &  =k!\sum\limits_{m=0}%
^{\infty}S_{2}\left(  m,k|\mu,\lambda,x\right)  \frac{t^{m}}{m!}. \label{stx}%
\end{align}
with the notations
\begin{align*}
S_{1}\left(  m,k|\mu,\lambda,x\right)   &  =S^{1}\left(  m,k\right)  =S\left(
m,k;\mu,\lambda,x\right)  ,\\
S_{2}\left(  m,k|\mu,\lambda,x\right)   &  =S^{2}\left(  m,k\right)  =S\left(
m,k;\lambda,\mu,-x\right)
\end{align*}
and the convention $S_{1}\left(  m,k|\mu,\lambda,x\right)  =S_{2}\left(
m,k|\mu,\lambda,x\right)  =0$ when $k>m.$

As Hsu and Shiue pointed out, the definitions or generating functions
generalize various Stirling-type numbers studied previously, such as;

(i) $\left\{  S_{1}\left(  m,k|1,0,0\right)  ,S_{2}\left(  m,k|1,0,0\right)
\right\}  =\left\{  s\left(  m,k\right)  ,S\left(  m,k\right)  \right\}
=\left\{  \left(  -1\right)  ^{m-k}%
\genfrac{[}{]}{0pt}{}{m}{k}%
,%
\genfrac{\{}{\}}{0pt}{}{m}{k}%
\right\}  $

$\qquad=\left\{  S_{1}\left(  m,k\right)  ,S_{2}\left(  m,k\right)  \right\}
$ are the Stirling numbers of the first kind and of the\ second kind,
respectively (\cite[Ch. 6]{Knuth}).

(ii) $\left\{  S_{1}\left(  m,k|1,\lambda,-x\right)  ,S_{2}\left(
m,k|1,\lambda,-x\right)  \right\}  =\left\{  \left(  -1\right)  ^{m-k}%
S_{1}\left(  m,k,x+\lambda|\lambda\right)  ,S_{2}\left(  m,k,x|\lambda\right)
\right\}  $

are Howard's degenerate weighted Stirling numbers of both kinds (\cite{18}).

(iii) $\left\{  S_{1}\left(  m,k|1,0,-x\right)  ,S_{2}\left(
m,k|1,0,-x\right)  \right\}  =\left\{  \left(  -1\right)  ^{m-k}R_{1}\left(
m,k,x\right)  ,R_{2}\left(  m,k,x\right)  \right\}  $ are Carlitz's weighted
Stirling numbers of both kinds (\cite{19}).

(iv) $\left\{  S_{1}\left(  m,k|-1,0,r\right)  ,S_{2}\left(
m,k|-1,0,r\right)  \right\}  =\left\{
\genfrac{[}{]}{0pt}{}{m+r}{k+r}%
_{r},\left(  -1\right)  ^{m-k}%
\genfrac{\{}{\}}{0pt}{}{m+r}{k+r}%
_{r}\right\}  $ are the $r-$Stirling numbers of both kinds (\cite{Broder}).

(v) $\left\{  S_{1}\left(  m,k|1,\lambda,0\right)  ,S_{2}\left(
m,k|1,\lambda,0\right)  \right\}  =\left\{  \left(  -1\right)  ^{m-k}%
S_{1}\left(  m,k|\lambda\right)  ,S_{2}\left(  m,k|\lambda\right)  \right\}  $
are Carlitz's degenerate Stirling numbers of both kinds (\cite{8}).

(vi) $\left\{  S_{1}\left(  m,k|-1,1,0\right)  ,S_{2}\left(
m,k|-1,1,0\right)  \right\}  =\left\{  L\left(  m,k\right)  ,\left(
-1\right)  ^{m-k}L\left(  m,k\right)  \right\}  ,$ where \linebreak$L\left(
m,k\right)  =\frac{m!}{k!}\binom{m-1}{k-1}$ are the Lah numbers.

The list above may not be complete. The combinatorial interpretations of
(i)--(iv) can be found in \cite{Broder, 19, Knuth, 18}.

From (\ref{stx1}) and (\ref{stx}), it follows that
\begin{equation}
S_{1}\left(  m,k|\mu,\lambda,x\right)  =\mu^{m-k}S_{1}\left(  m,k|1,\frac
{\lambda}{\mu},\frac{x}{\mu}\right)  \text{ and }S_{2}\left(  m,k|\mu
,\lambda,x\right)  =\mu^{m-k}S_{2}\left(  m,k|1,\frac{\lambda}{\mu},\frac
{x}{\mu}\right)  \label{8}%
\end{equation}
for $\mu\not =0.$ Letting $\lambda=0$ and $x=0$ in (\ref{8}), we have
\begin{equation}
S_{1}\left(  m,k|\mu,0,0\right)  =\mu^{m-k}S_{1}\left(  m,k\right)  \text{ and
}S_{2}\left(  m,k|\mu,0,0\right)  =\mu^{m-k}S_{2}\left(  m,k\right)  .
\label{10}%
\end{equation}

\subsection{Stirling matrices of the first type}

Let $s_{n}\left[  \mu,\lambda,x\right]  $\ and $S_{n}\left[  \mu
,\lambda,x\right]  $\ be the $n\times n$ matrices defined by%
\[
\left(  s_{n}\left[  \mu,\lambda,x\right]  \right)  _{i,j}=\left\{
\begin{array}
[c]{ll}%
S_{1}\left(  i,j|\mu,\lambda,x\right)  & \text{, if }i\geq j\geq1\\
0 & \text{, if }i<j
\end{array}
\right.
\]
and%
\[
\left(  \mathcal{S}_{n}\left[  \mu,\lambda,x\right]  \right)  _{i,j}=\left\{
\begin{array}
[c]{ll}%
S_{2}\left(  i,j|\mu,\lambda,x\right)  & \text{, if }i\geq j\geq1\\
0 & \text{, if }i<j
\end{array}
\right.
\]
which we call generalized Stirling matrices of the first type. Then the
relation (cf. \cite[Eq. (3)]{17})%
\begin{equation}
\sum\limits_{k=j}^{i}S_{1}\left(  i,k|\mu,\lambda,x\right)  S_{2}\left(
k,j|\mu,\lambda,x\right)  =\sum\limits_{k=j}^{i}S_{2}\left(  i,k|\mu
,\lambda,x\right)  S_{1}\left(  k,j|\mu,\lambda,x\right)  =\delta_{i,j}
\label{9}%
\end{equation}
yields $\mathcal{S}_{n}^{-1}\left[  \mu,\lambda,x\right]  =s_{n}\left[
\mu,\lambda,x\right]  .$ From (\ref{10}), it is seen that $s_{n}\left[
\mu,0,0\right]  $ and $\mathcal{S}_{n}\left[  \mu,0,0\right]  $ are the
Stirling\textit{\ }matrices $\mathcal{S}_{n}^{-1}\left[  \mu\right]  $ and
$\mathcal{S}_{n}\left[  \mu\right]  $ defined in \cite[p. 57]{6}.

Differentiate both sides of (\ref{stx}) with respect to $t$ (with $x=0$) to
get
\begin{align*}
&  \sum\limits_{m=0}^{\infty}S_{2}(m+1,j|\mu,\lambda,0)\frac{t^{m}}{m!}%
=\frac{1}{(j-1)!}(1+\lambda t)^{\left(  \mu-\lambda\right)  /\lambda}\left[
\frac{(1+\lambda t)^{\mu/\lambda}-1}{\mu}\right]  ^{j-1}\\
&  \quad=\sum\limits_{m=0}^{\infty}(\mu-\lambda|\lambda)_{m}\frac{t^{m}}%
{m!}\sum\limits_{m=0}^{\infty}S_{2}(m,j-1|\mu,\lambda,0)\frac{t^{m}}{m!}\\
&  \quad=\sum\limits_{m=0}^{\infty}\left[  \sum\limits_{k=j-1}^{m}\dbinom
{m}{k}(\mu-\lambda|\lambda)_{m-k}S_{2}(k,j-1|\mu,\lambda,0)\right]
\frac{t^{m}}{m!}.
\end{align*}
Thus,
\[
S_{2}(m+1,j|\mu,\lambda,0)=\sum\limits_{k=j-1}^{m}\dbinom{m}{k}(\mu
-\lambda|\lambda)_{m-k}S_{2}(k,j-1|\mu,\lambda,0).
\]
Putting $m=i-1$ gives
\begin{equation}
S_{2}\left(  i,j|\mu,\lambda,0\right)  =\sum\limits_{k=j}^{i}\dbinom{i-1}%
{k-1}(\mu-\lambda|\lambda)_{i-k}S_{2}(k-1,j-1|\mu,\lambda,0) \label{4}%
\end{equation}
which yields%
\begin{equation}
\mathcal{S}_{n}\left[  \mu,\lambda,0\right]  =\mathcal{P}_{n}\left[
\lambda,\mu-\lambda\right]  \left(  \left[  1\right]  \oplus\mathcal{S}%
_{n-1}\left[  \mu,\lambda,0\right]  \right)  . \label{19}%
\end{equation}

Note that (\ref{4}) reduces to the well known \textit{vertical} recurrence
relation%
\[
S_{2}(i,j)=\sum\limits_{k=j}^{i}\dbinom{i-1}{k-1}S_{2}(k-1,j-1)
\]
by letting $\lambda=0$ and $\mu\not =0$. The counterpart of (\ref{4}) is
\[
S_{1}\left(  i,j|\mu,\lambda,0\right)  =\sum\limits_{k=j}^{i}\dbinom{k-1}%
{j-1}S_{1}(i-1,k-1|\mu,\lambda,0)(\lambda-\mu|\lambda)_{k-j}.
\]
We also have
\begin{align*}
\dbinom{i-1}{j-1}\left(  \mu-\lambda|\lambda\right)  _{i-j}  &  =\sum
\limits_{k=j}^{i}S_{2}\left(  i,k|\mu,\lambda,0\right)  S_{1}\left(
k-1,j-1|\mu,\lambda,0\right)  ,\\
\dbinom{i-1}{j-1}\left(  \lambda-\mu|\lambda\right)  _{i-j}  &  =\sum
\limits_{k=j}^{i}S_{2}\left(  i-1,k-1|\mu,\lambda,0\right)  S_{1}\left(
k,j|\mu,\lambda,0\right)  .
\end{align*}

Furthermore, in consequence of (\ref{19}) we have the following factorization
of the matrix $\mathcal{S}_{n}\left[  \mu,\lambda,0\right]  $:%
\[
\mathcal{S}_{n}\left[  \mu,\lambda,0\right]  =Q_{n}\left[  \lambda,\mu
-\lambda\right]  Q_{n-1}\left[  \lambda,\mu-\lambda\right]  \cdots
Q_{1}\left[  \lambda,\mu-\lambda\right]  ,
\]
where $Q_{k}\left[  \lambda,x\right]  =I_{n-k}\oplus\mathcal{P}_{k}\left[
\lambda,x\right]  $, $1\leq k\leq n-1$ and $Q_{n}\left[  \lambda,x\right]
=\mathcal{P}_{n}\left[  \lambda,x\right]  .$

\subsection{Stirling matrices of the second type}

Let us define the second type generalized Stirling\ matrices $\mathcal{G}%
_{n,h}\left[  1,\lambda,x\right]  $\ and $g_{n,h}\left[  1,\lambda,x\right]
$\ of order $n$\ by%
\[
\left(  \mathcal{G}_{n,h}\left[  1,\lambda,x\right]  \right)  _{i,j}=\left\{
\begin{array}
[c]{ll}%
\binom{i-1}{j-1}\binom{i-h}{j-h}^{-1}S_{2}\left(  i-h,j-h|1,\lambda,x\right)
& \text{, if }i>j\geq1\text{ and }j\geq h\\
1 & \text{, if }i=j\\
0 & \text{, otherwise}%
\end{array}
\right.
\]
and
\[
\left(  g_{n,h}\left[  1,\lambda,x\right]  \right)  _{i,j}=\left\{
\begin{array}
[c]{ll}%
\binom{i-1}{j-1}\binom{i-h}{j-h}^{-1}S_{1}\left(  i-h,j-h|1,\lambda,x\right)
& \text{, if }i>j\geq1\text{ and }j\geq h\\
1 & \text{, if }i=j\\
0 & \text{, otherwise}.
\end{array}
\right.
\]
It is obvious from (\ref{9}) that $g_{n,h}\left[  1,\lambda,x\right]  =\left(
\mathcal{G}_{n,h}\left[  1,\lambda,x\right]  \right)  ^{-1}$. We have

\begin{theorem}
\label{th bs}%
\begin{align*}
\left(  \mathcal{G}_{n,h}\left[  1,\lambda,-x\right]  \right)  _{i,j}  &
=\left(  \mathcal{B}_{n}^{\left(  h\right)  }\left[  \lambda,x-y\right]
\mathcal{G}_{n,0}\left[  1,\lambda,-y\right]  \right)  _{i,j}=\left(
\mathcal{G}_{n,0}\left[  1,\lambda,-y\right]  \mathcal{L}_{n}^{\left(
-h\right)  }\left[  \lambda,x-y\right]  \right)  _{i,j},\\
\left(  g_{n,h}\left[  1,\lambda,-x\right]  \right)  _{i,j}  &  =\left(
g_{n,0}\left[  1,\lambda,-y\right]  \mathcal{B}_{n}^{\left(  -h\right)
}\left[  \lambda,y-x\right]  \right)  _{i,j}=\left(  \mathcal{L}_{n}^{\left(
h\right)  }\left[  \lambda,y-x\right]  g_{n,0}\left[  1,\lambda,-y\right]
\right)  _{i,j}.
\end{align*}
for $j\geq h.$ In particular,%
\begin{align*}
\mathcal{P}_{n}\left[  \lambda,x-y\right]   &  =\mathcal{G}_{n,0}\left[
1,\lambda,-x\right]  g_{n,0}\left[  1,\lambda,-y\right]  ,\\
\mathcal{P}_{n}\left[  1,y-x\right]   &  =g_{n,0}\left[  1,\lambda,-x\right]
\mathcal{G}_{n,0}\left[  1,\lambda,-y\right]  .
\end{align*}

\end{theorem}

\begin{proof}
By (\ref{1}) and (\ref{stx}), we have%
\begin{align*}
&  \left(  \frac{t}{(1+\lambda t)^{1/\lambda}-1}\right)  ^{h}\frac{(1+\lambda
t)^{x/\lambda}}{j!}\left[  (1+\lambda t)^{1/\lambda}-1\right]  ^{j}\\
&  \quad=\sum_{m=j}^{\infty}\left(  \sum_{k=j}^{m}\binom{m}{k}\frac{1}%
{m!}\beta_{m-k}^{\left(  h\right)  }(\lambda,x-y)S_{2}\left(  k,j|1,\lambda
,-y\right)  \right)  t^{m}\\
&  \quad=\sum_{m=j}^{\infty}\left(  \frac{\left(  j-h\right)  !}{j!\left(
m-h\right)  !}S_{2}\left(  m-h,j-h|1,\lambda,-x\right)  \right)  t^{m}\\
&  \quad=\frac{t^{h}}{j!}(1+\lambda t)^{x/\lambda}\left[  (1+\lambda
t)^{1/\lambda}-1\right]  ^{j-h}%
\end{align*}
for an integer $h$ and $j\geq h.$ Then
\begin{equation}
\frac{\binom{i}{j}}{\binom{i-h}{j-h}}S_{2}\left(  i-h,j-h|1,\lambda,-x\right)
=\sum_{k=j}^{i}\binom{i}{k}\beta_{i-k}^{\left(  h\right)  }(\lambda
,x-y)S_{2}\left(  k,j|1,\lambda,-y\right)  \label{5}%
\end{equation}
or%
\[
\frac{\binom{i-1}{j-1}}{\binom{i-h}{j-h}}S_{2}\left(  i-h,j-h|1,\lambda
,-x\right)  =\sum_{k=j}^{i}\binom{i-1}{k-1}\beta_{i-k}^{\left(  h\right)
}(\lambda,x-y)\frac{j}{k}S_{2}\left(  k,j|1,\lambda,-y\right)  ,
\]
which gives $\left(  \mathcal{G}_{n,h}\left[  1,\lambda,-x\right]  \right)
_{i,j}=\left(  \mathcal{B}_{n}^{\left(  h\right)  }\left[  \lambda,x-y\right]
\mathcal{G}_{n,0}\left[  1,\lambda,-y\right]  \right)  _{i,j}.$

It can be seen from (\ref{8}) and generating functions that{\Large \ }%
\begin{equation}
S_{2}\left(  i,j|1,\frac{1}{\lambda},\frac{x}{\lambda}\right)  =\left(
\frac{1}{\lambda}\right)  ^{i-j}S_{1}\left(  i,j|1,\lambda,-x\right)  .
\label{50}%
\end{equation}
Thus, taking into account (\ref{1b2}) and (\ref{50}), replace $\left(
\lambda,x,y\right)  $\ by $\left(  \dfrac{1}{\lambda},-\dfrac{x}{\lambda
},-\dfrac{y}{\lambda}\right)  $\ in (\ref{5}) to get%
\begin{equation}
\frac{\binom{i}{j}}{\binom{i-h}{j-h}}S_{1}\left(  i-h,j-h|1,\lambda,-x\right)
=\sum_{k=j}^{i}\binom{i}{k}\alpha_{i-k}^{\left(  h\right)  }(\lambda
,y-x)S_{1}\left(  k,j|1,\lambda,-y\right)  \label{5a}%
\end{equation}
or
\[
\frac{\binom{i-1}{j-1}}{\binom{i-h}{j-h}}S_{1}\left(  i-h,j-h|1,\lambda
,-x\right)  =\sum_{k=j}^{i}\binom{i-1}{k-1}\alpha_{i-k}^{\left(  h\right)
}(\lambda,y-x)\frac{j}{k}S_{1}\left(  k,j|1,\lambda,-y\right)
\]
which gives $\left(  g_{n,h}\left[  1,\lambda,-x\right]  \right)
_{i,j}=\left(  \mathcal{L}_{n}^{\left(  h\right)  }\left[  \lambda,y-x\right]
g_{n,0}\left[  1,\lambda,-y\right]  \right)  _{i,j}.$
\end{proof}

The theorem above shows that the Bernoulli polynomials can be expressed in
terms of the Stirling numbers.

\begin{corollary}
\label{cor bst1}For $i\geq j\geq0$ and $j\geq h,$ we have
\begin{align}
\dbinom{i}{j}\beta_{i-j}^{\left(  h\right)  }\left(  \lambda,x-y\right)   &
=\sum\limits_{k=j}^{i}\frac{\binom{i}{k}}{\binom{i-h}{k-h}}S_{2}\left(
i-h,k-h|1,\lambda,-x\right)  S_{1}\left(  k,j|1,\lambda,-y\right)
,\label{6}\\
\dbinom{i}{j}\alpha_{i-j}^{\left(  h\right)  }\left(  \lambda,y-x\right)   &
=\sum\limits_{k=j}^{i}\frac{\binom{i}{k}}{\binom{i-h}{k-h}}S_{1}\left(
i-h,k-h|1,\lambda,-x\right)  S_{2}\left(  k,j|1,\lambda,-y\right)  .
\label{6a}%
\end{align}
In particular,
\begin{align*}
\beta_{i}\left(  \lambda,x\right)   &  =\sum\limits_{k=0}^{i}\frac{1}%
{k+1}S_{2}\left(  i,k|1,\lambda,-x\right)  \left(  \lambda-1\right)  _{k},\\
\alpha_{i}\left(  \lambda,x\right)   &  =\sum\limits_{k=0}^{i}\frac{1}%
{k+1}S_{1}\left(  i,k|1,\lambda,x\right)  \left(  1|\lambda\right)  _{k+1}.
\end{align*}

\end{corollary}

Note that we can equally well write the result (\ref{6}) in the form
\begin{align}
&  \binom{i-m}{j-m}\beta_{i-j}^{\left(  h-m\right)  }\left(  \lambda
,x-y\right) \nonumber\\
&  \qquad=\sum\limits_{k=j}^{i}\frac{\binom{i-m}{k-m}}{\binom{i-h}{k-h}}%
S_{2}\left(  i-h,k-h|1,\lambda,-x\right)  S_{1}\left(  k-m,j-m|1,\lambda
,-y\right)  \label{2}%
\end{align}
for $i\geq j\geq0$ and $j\geq\max\left\{  h,m\right\}  $ because of\textbf{\ }%
\begin{align*}
\mathcal{G}_{n,h}\left[  1,\lambda,-x\right]  g_{n,m}\left[  1,\lambda
,-y\right]   &  =\mathcal{G}_{n,h}\left[  1,\lambda,-x\right]  g_{n,0}\left[
1,\lambda,-z\right]  \mathcal{G}_{n,0}\left[  1,\lambda,-z\right]
g_{n,m}\left[  1,\lambda,-y\right] \\
&  =\mathcal{B}_{n}^{\left(  h\right)  }\left[  \lambda,x-z\right]
\mathcal{B}_{n}^{\left(  -m\right)  }\left[  \lambda,z-y\right]
=\mathcal{B}_{n}^{\left(  h-m\right)  }\left[  \lambda,x-y\right]  .
\end{align*}
Set $h=y=\lambda=0$ and $-m=l\geq0$ to compute any positive (integer) order of
Bernoulli polynomials
\[
\binom{i}{j}\binom{j+l}{j}^{-1}B_{i-j}^{\left(  l\right)  }\left(  x\right)
=\sum\limits_{k=j}^{i}\binom{l+k}{l}^{-1}R_{2}\left(  i,k,x\right)
S_{1}\left(  l+k,l+j\right)
\]
and in particular
\[
B_{i}^{\left(  i\right)  }=\sum\limits_{k=0}^{i}\binom{i+k}{i}^{-1}%
S_{2}\left(  i,k\right)  S_{1}\left(  i+k,i\right)  .
\]

By (\ref{8}) and Corollary \ref{cor bst1}, we have the generalized
orthogonality relations
\begin{align*}
\dbinom{i}{j}\left(  x-y|\lambda\right)  _{i-j}  &  =\sum\limits_{k=j}%
^{i}S_{2}\left(  i,k|\mu,\lambda,-x\right)  S_{1}\left(  k,j|\mu
,\lambda,-y\right)  ,\\
\dbinom{i}{j}\left(  y-x|\mu\right)  _{i-j}  &  =\sum\limits_{k=j}^{i}%
S_{1}\left(  i,k|\mu,\lambda,-x\right)  S_{2}\left(  k,j|\mu,\lambda
,-y\right)  .
\end{align*}

Theorem \ref{th bs} also entails following.

\begin{corollary}
\label{cor bst2}For $i\geq j\geq0$ and $j\geq h,$ we have%
\begin{align}
\frac{\binom{i}{j}}{\binom{i-h}{j-h}}S_{2}\left(  i-h,j-h|1,\lambda,-x\right)
&  =\sum_{k=j}^{i}\binom{k}{j}S_{2}\left(  i,k|1,\lambda,-y\right)
\alpha_{k-j}^{\left(  -h\right)  }(\lambda,x-y),\label{7a}\\
\frac{\binom{i}{j}}{\binom{i-h}{j-h}}S_{1}\left(  i-h,j-h|1,\lambda,-x\right)
&  =\sum_{k=j}^{i}\binom{k}{j}S_{1}\left(  i,k|1,\lambda,-y\right)
\beta_{k-j}^{\left(  -h\right)  }(\lambda,y-x). \label{7}%
\end{align}
In particular,%
\begin{align*}
\frac{j+1}{i+1}S_{2}\left(  i+1,j+1|1,\lambda,-x\right)   &  =\sum_{k=j}%
^{i}\binom{k}{j}S_{2}\left(  i,k|\lambda\right)  \alpha_{k-j}(\lambda,x),\\
\frac{j+1}{i+1}S_{1}\left(  i+1,j+1|1,\lambda,x\right)   &  =\sum_{k=j}%
^{i}\binom{k}{j}\left(  -1\right)  ^{i-k}S_{1}\left(  i,k|\lambda\right)
\beta_{k-j}(\lambda,x).
\end{align*}

\end{corollary}

In consequence of Theorem \ref{teo pg} and Theorem \ref{th bs} we have

\begin{corollary}%
\begin{align*}
\mathcal{G}_{n,0}\left[  1,\lambda,-x\right]   &  =G_{n}\left[  \lambda
,x\right]  G_{n-1}\left[  \lambda,x\right]  \cdots G_{1}\left[  \lambda
,x\right]  \mathcal{G}_{n,0}\left[  1,\lambda,0\right] \\
&  =\mathcal{G}_{n,0}\left[  1,\lambda,0\right]  G_{n}\left[  1,x\right]
G_{n-1}\left[  1,x\right]  \cdots G_{1}\left[  1,x\right]  .
\end{align*}
\bigskip
\end{corollary}

It is evident from definitions that
\begin{align}
\beta_{m}^{\left(  -h\right)  }(\lambda,x)  &  =\dbinom{m+h}{h}^{-1}%
S_{2}\left(  m+h,h|1,\lambda,-x\right)  ,\label{14}\\
\alpha_{m}^{\left(  -h\right)  }(\lambda,x)  &  =\dbinom{m+h}{h}^{-1}%
S_{1}\left(  m+h,h|1,\lambda,x\right)  \label{14a}%
\end{align}
for an integer $h\geq0$. Therefore, (\ref{5}), (\ref{5a}), (\ref{7a}) and
(\ref{7}) may be specified according as $h$ is negative or positive :
\begin{align}
&  \binom{h+j}{j}S_{2}\left(  m,j+h|1,\lambda,-\left(  x+y\right)  \right)
\nonumber\\
&  \qquad\qquad=\sum_{k=j}^{m-h}\binom{m}{k}S_{2}\left(  m-k,h|1,\lambda
,-x\right)  S_{2}\left(  k,j|1,\lambda,-y\right)  ,\text{ }m=i+h,\label{s2s}\\
&  \binom{h+j}{j}S_{1}\left(  m,j+h|1,\lambda,-\left(  x+y\right)  \right)
\nonumber\\
&  \qquad\qquad=\sum_{k=j}^{m-h}\binom{m}{k}S_{1}\left(  m-k,h|1,\lambda
,-x\right)  S_{1}\left(  k,j|1,\lambda,-y\right)  ,\text{ }m=i+h, \label{s1s}%
\end{align}%
\begin{align}
&  \frac{\binom{i}{j}}{\binom{i+h}{j+h}}S_{2}\left(  i+h,j+h|1,\lambda
,-\left(  x+y\right)  \right)  =\sum_{k=j}^{i}\binom{k}{j}\alpha
_{k-j}^{\left(  h\right)  }(\lambda,x)S_{2}\left(  i,k|1,\lambda,-y\right)
,\label{7a1}\\
&  \frac{\binom{i}{j}}{\binom{i+h}{j+h}}S_{1}\left(  i+h,j+h|1,\lambda
,-\left(  y-x\right)  \right)  =\sum_{k=j}^{i}\binom{k}{j}\beta_{k-j}^{\left(
h\right)  }(\lambda,x)S_{1}\left(  i,k|1,\lambda,-y\right)  \label{71}%
\end{align}
for $i\geq j\geq0$ and arbitrary integer\ $h\geq0$, and
\begin{align}
&  \frac{\binom{i}{j}}{\binom{i-h}{j-h}}S_{2}\left(  i-h,j-h|1,\lambda
,-\left(  x+y\right)  \right)  =\sum_{k=j}^{i}\binom{i}{k}\beta_{i-k}^{\left(
h\right)  }(\lambda,x)S_{2}\left(  k,j|1,\lambda,-y\right)  ,\label{51}\\
&  \frac{\binom{i}{j}}{\binom{i-h}{j-h}}S_{1}\left(  i-h,j-h|1,\lambda
,-\left(  y-x\right)  \right)  =\sum_{k=j}^{i}\binom{i}{k}\alpha
_{i-k}^{\left(  h\right)  }(\lambda,x)S_{1}\left(  k,j|1,\lambda,-y\right)
,\label{5a1}\\
&  \binom{i}{h}S_{2}\left(  i-h,j-h|1,\lambda,-\left(  y-x\right)  \right)
\nonumber\\
&  \qquad\qquad=\sum_{k=j}^{i}\binom{k}{j-h}S_{2}\left(  i,k|1,\lambda
,-y\right)  S_{1}\left(  k-j+h,h|1,\lambda,-x\right)  ,\label{7a2}\\
&  \binom{i}{h}S_{1}\left(  i-h,j-h|1,\lambda,-\left(  y-x\right)  \right)
\nonumber\\
&  \qquad\qquad=\sum_{k=j}^{i}\binom{k}{j-h}S_{1}\left(  i,k|1,\lambda
,-y\right)  S_{2}\left(  k-j+h,h|1,\lambda,-x\right)  \label{72}%
\end{align}
for $i\geq j\geq h\geq0.$

By (\ref{8}), identities (\ref{s2s}) and (\ref{s1s}) coincide the "addition
theorems" (cf. \cite[Corollary 2]{17}), and (\ref{7a2}) and (\ref{72}) are
still valid for $S_{1}\left(  n,m|\mu,\lambda,x\right)  $ and $S_{2}\left(
n,m|\mu,\lambda,x\right)  .$

\begin{remark}
(\ref{51}) and (\ref{5a1}) reduce to \cite[Eqs. (3.2) and (5.3) ]{19} and
\cite[Theorems 12 and 14]{Broder} taking account of (\ref{11}) and (\ref{11a})
given later$.$ We also have identities that appear in \cite[Theorem 5.1]{A-D}
and \cite[Theorem 4.4 and their results]{23}$.$
\end{remark}

\begin{remark}
From Corollary \ref{cor bst1} we have identities given in \cite[Theorem
25]{Broder}, \cite[Eqs. (6.3) and (6.5)]{19} and\ \cite[Eq. (6.99)]{Knuth}
(and \cite[Corollary 3.3]{11})$.$ Additionally, by $\left(  s+r\right)
_{r}\left(  s\right)  _{k-r}=\left(  s+r\right)  _{k},$ (\ref{2}) gives
\cite[Eq. (2.89)]{25} for $x=r=j=h=m,$ $y=-s$ and $\lambda=0.$
\end{remark}

\begin{remark}
As a special case of Corollary \ref{cor bst2}, we have Carlitz's results
appearing in \cite[Eqs. (3.3), (3.4), (3.25), (5.2) and (5.8)]{19}.
Furthermore, \cite[Theorem 5]{24} is a special case of (\ref{71}) for
$\lambda=y=j=0$\ and $h=1.$
\end{remark}

\begin{remark}
(\ref{51}) gives \cite[Theorem 4.1]{3} for $h=j,$ $x=y=0$ and $i=m+j$.
\end{remark}

\begin{remark}
The identities given by \cite[Eqs. (6.15), (6.17), (6.21), (6.24), (6.25),
(6.28) and (6.29)]{Knuth} are special cases of (\ref{s2s}) for $h=y=0$ and
$x=1,$ (\ref{51}) for $h=0,$ $y=1$ and $x=-1$, (\ref{5a1}) for $h=y=0$ and
$x=-1$, (\ref{6}) for $h=y=0$ and $x=1$, (\ref{6a}) for $h=y=0$ and $x=1$,
(\ref{s2s}) for $x=y=0,$ (\ref{s1s}) for $x=y=0$ and in all cases $\lambda=0$, respectively.
\end{remark}

\begin{remark}
Broder noted that Nielsen "Traite Elementaire des Nombres de Bernoulli,
Gauthier-Villars, Paris, 1923, Ch. 12" developed a large number of formulae
relating $R_{2}(n,m,x)=S_{2}\left(  n,m|1,0,-x\right)  $\ to the Bernoulli and
Euler polynomials. (Nielsen's notation is $A_{m}^{n}\left(  x\right)
=m!R_{2}(n,m,x)$). This note reveals that identities follow from our results
for $\lambda=0$\ may be studied by Nielsen with probably different notations.
\end{remark}

\section{Applications of the results in Subsection 4.2}

By using the results presented in (\ref{s2s})-(\ref{72}) and Corollary
\ref{cor bst1} several identities for some related number sequences can be
deduced for the special values of $h$, $\mu$, $\lambda$, $x$ and $y,$ and from
the identities $S_{1}\left(  i,j|1,\lambda,-x\right)  =\left(  -1\right)
^{i-j}S_{1}\left(  i+1,j+1|\lambda\right)  $ and $S_{2}\left(  i,j|1,\lambda
,-x\right)  =S_{2}\left(  i+1,j+1|\lambda\right)  $ for $x=1-\lambda$. In
previous section we mentioned some of them. In this section we partly specify
identities given by (\ref{s2s})-(\ref{72}) and Corollary \ref{cor bst1}.

\subsection{Carlitz's weighted Stirling numbers}

In this part we give results involving Carlitz's weighted Stirling numbers and
Bernoulli polynomials. For $\lambda=0$, we have from (\ref{6})
\begin{align*}
\dbinom{i}{j}B_{i-j}(x)  &  =\sum\limits_{k=j}^{i}\frac{i}{k}R_{2}\left(
i-1,k-1,x\right)  S_{1}\left(  k,j\right)  ,\text{ for }h=1,\text{ }y=0,\\
B_{i}(x)  &  =\sum\limits_{k=0}^{i}\left(  -1\right)  ^{k}\frac{k!}{k+1}%
R_{2}\left(  i,k,x\right)  ,\text{ for }h=j=1,\text{ }y=0,
\end{align*}
from (\ref{6a})%
\begin{align*}
i!b_{i}\left(  -x\right)   &  =\sum\limits_{k=0}^{i}\frac{\left(  -1\right)
^{i-k}}{k+1}R_{1}\left(  i,k,x\right)  ,\text{ for }h=j=1,\text{ }y=0,\\
\left(  -1\right)  ^{i-1}\left(  i-1\right)  !  &  =\sum\limits_{k=0}%
^{i}\left(  k+1\right)  S_{1}\left(  i+1,k+1\right)  ,\text{ for
}h=-1,y=1,x=j=0,
\end{align*}
by (\ref{14a}) and%
\begin{equation}
S_{1}\left(  m,1|1,0,1\right)  =S_{1}\left(  m,1\right)  +mS_{1}\left(
m-1,1\right)  =\left\{
\begin{array}
[c]{ll}%
1, & m=1,\\
\left(  -1\right)  ^{m}\left(  m-2\right)  !, & m>1,
\end{array}
\right.  \label{17}%
\end{equation}
which can be held from (\ref{5a1}) for $\lambda=h=y=0,$ $j=1$ and $x=1.$

From (\ref{7a1})
\begin{align*}
\frac{i!}{\left(  i+h\right)  !}R_{2}\left(  i+h,j+h,x\right)   &  =\frac
{1}{\left(  j+h\right)  !}\sum_{k=j}^{i}S_{2}\left(  i,k\right)
b_{k-j}^{\left(  h\right)  }(x)k!,\text{ for }y=0,\\
\frac{\left(  x+1\right)  ^{i+1}-x^{i+1}}{i+1}  &  =\sum_{k=0}^{i}S_{2}\left(
i,k\right)  b_{k}(x)k!,\text{ for }h=1,j=y=0,
\end{align*}
and from (\ref{71})
\begin{align*}
\left(  -1\right)  ^{i-j}\frac{\binom{i}{j}}{\binom{i+h}{j+h}}R_{1}\left(
i+h,j+h,-x\right)   &  =\sum_{k=j}^{i}\binom{k}{j}S_{1}\left(  i,k\right)
B_{k-j}^{\left(  h\right)  }(x),\text{ for }y=0,\\
\frac{\binom{i}{j}}{\binom{i+h}{j+h}}S_{1}\left(  i+h+1,j+h+1\right)   &
=\sum_{k=j}^{i}\binom{k}{j}S_{1}\left(  i+1,k+1\right)  B_{k-j}^{\left(
h\right)  },\text{ for }y=1,x=0.
\end{align*}

\subsection{$r$-Stirling numbers}

Setting $\mu=-1$ in (\ref{8}) we have
\begin{align}
S_{1}\left(  m,k|1,-\lambda,-x\right)   &  =\left(  -1\right)  ^{m-k}%
S_{1}\left(  m,k|-1,\lambda,x\right)  ,\label{11}\\
S_{2}\left(  m,k|1,-\lambda,-x\right)   &  =\left(  -1\right)  ^{m-k}%
S_{2}\left(  m,k|-1,\lambda,x\right)  . \label{11a}%
\end{align}
So that
\[
\hspace{-0.3in}S_{1}\left(  m,k|1,0,-r\right)  =\left(  -1\right)  ^{m-k}%
\genfrac{[}{]}{0pt}{}{m+r}{k+r}%
_{r}\text{ and }S_{2}\left(  m,k|1,0,-r\right)  =%
\genfrac{\{}{\}}{0pt}{}{m+r}{k+r}%
_{r}%
\]
for $\lambda=0$ and nonnegative integer $x=r$\textbf{. }Then the results
presented in Subsection 4.2 can be specialized in terms of $r$-Stirling
numbers for $\lambda=0$ and integers $x$ and $y.$

In the first place note that
\[%
\genfrac{[}{]}{0pt}{}{m}{k}%
_{0}=%
\genfrac{[}{]}{0pt}{}{m}{k}%
_{1}=%
\genfrac{[}{]}{0pt}{}{m}{k}%
,\text{ \ \ \ }%
\genfrac{\{}{\}}{0pt}{}{m}{k}%
_{0}=%
\genfrac{\{}{\}}{0pt}{}{m}{k}%
_{1}=%
\genfrac{\{}{\}}{0pt}{}{m}{k}%
.
\]

The following are special cases of (\ref{72}) for $\lambda=0,$ nonnegative
integers $y=p,$ $y-x=r$ and $j\geq h:$
\begin{align}
&
\genfrac{[}{]}{0pt}{}{i+r}{j+r}%
_{r}=\sum_{k=j}^{i}\binom{k}{j}%
\genfrac{[}{]}{0pt}{}{i+p}{k+p}%
_{p}(r-p)^{k-j},\text{ for }h=0\text{,}\label{72-1}\\
&
\genfrac{[}{]}{0pt}{}{i+r}{j+r}%
_{r}=\frac{1}{i+1}\sum_{k=j}^{i}\binom{k+1}{j}%
\genfrac{[}{]}{0pt}{}{i+r}{k+r}%
_{r-1},\text{ for }h=1,\text{ }p=r-1,\nonumber\\
&  \binom{i}{j}\left\langle p+j\right\rangle _{i-j}=\sum_{k=j}^{i}%
\genfrac{[}{]}{0pt}{}{i+p}{k+p}%
_{p}%
\genfrac{\{}{\}}{0pt}{}{k}{j}%
,\text{ for }h=j,\text{ }r=p+j, \nonumber\label{72-2}%
\end{align}
where $\left\langle m\right\rangle _{k}=\left\{
\begin{array}
[c]{ll}%
m\left(  m+1\right)  \cdots\left(  m+k-1\right)  , & k>0,\\
1, & k=0
\end{array}
\right.  $ and we use that
\[
S_{2}\left(  k,m|1,0,m\right)  =\left(  -1\right)  ^{k-m}S_{2}\left(
k,m\right)  =\left(  -1\right)  ^{k-m}%
\genfrac{\{}{\}}{0pt}{}{k}{m}%
.
\]
Additionally, for $p=r-1$ and $p=r+1$ in (\ref{72-1}\textbf{) }we get
recurrences
\begin{align*}%
\genfrac{[}{]}{0pt}{}{i+r}{j+r}%
_{r}  &  =\sum_{k=j}^{i}\binom{k}{j}%
\genfrac{[}{]}{0pt}{}{i+r-1}{k+r-1}%
_{r-1},\\%
\genfrac{[}{]}{0pt}{}{i+r}{j+r}%
_{r}  &  =\sum_{k=j}^{i}\left(  -1\right)  ^{k-j}\binom{k}{j}%
\genfrac{[}{]}{0pt}{}{i+r+1}{k+r+1}%
_{r+1},
\end{align*}
respectively. The identities above reduce to \cite[Eq. (6.16) and
(6.18)]{Knuth} for $r=1$ and $r=0,$ respectively.

Some identities deduced from (\ref{7a2}),\ for $\lambda=0,$\ nonnegative
integers $y=p,$\ $y-x=r$\ and $j\geq h,$ are
\begin{align}
&
\genfrac{\{}{\}}{0pt}{}{i+r}{j+r}%
_{r}=\sum_{k=j}^{i}\binom{k}{j}%
\genfrac{\{}{\}}{0pt}{}{i+p}{k+p}%
_{p}\left(  r-p\right)  _{k-j},\text{ for }h=0,\nonumber\\
&
\genfrac{\{}{\}}{0pt}{}{i+p-1}{j+p-1}%
_{p-1}=\frac{1}{j!}\sum_{k=j}^{i}\left(  -1\right)  ^{k-j}%
\genfrac{\{}{\}}{0pt}{}{i+p}{k+p}%
_{p}k!,\text{ for }h=0,\text{ }r=p-1\label{7a2-1}\\
&
\genfrac{\{}{\}}{0pt}{}{i-1+r}{j-1+r}%
_{r}=\frac{1}{i\left(  j-1\right)  !}\sum_{k=j}^{i}\left(  -1\right)
^{k-j}\frac{k!}{k-j+1}%
\genfrac{\{}{\}}{0pt}{}{i+r}{k+r}%
_{r},\text{ for }h=1,\text{ }p=r\text{,}\label{7a2-2}\\
&
\genfrac{\{}{\}}{0pt}{}{i-1+r}{j-1+r}%
_{r}=\frac{1}{i\left(  j-1\right)  !}\sum_{k=j}^{i}\left(  -1\right)  ^{k-j}%
\genfrac{\{}{\}}{0pt}{}{i+r+1}{k+r+1}%
_{r+1}k!H_{k-j+1},\text{ for }h=1,\text{ }p=r+1\text{,}\nonumber
\end{align}
where $%
\genfrac{[}{]}{0pt}{}{m+1}{2}%
=m!H_{m}$ and
\[
H_{m}=1+\frac{1}{2}+\frac{1}{3}+\cdots+\frac{1}{m}.
\]
By (\ref{17}), we also have
\begin{equation}
\frac{i}{j}%
\genfrac{\{}{\}}{0pt}{}{i-1+r}{j-1+r}%
_{r}-%
\genfrac{\{}{\}}{0pt}{}{i+r-1}{j+r-1}%
_{r-1}=-\frac{1}{j!}\sum_{k=j+1}^{i}\frac{\left(  -1\right)  ^{k-j}%
\ k!}{\left(  k+1-j\right)  \left(  k-j\right)  }%
\genfrac{\{}{\}}{0pt}{}{i+r-1}{k+r-1}%
_{r-1} \label{7a2-3}%
\end{equation}
{\large \ }for $h=1$ and $p=r-1.$ Additionally, regular Stirling numbers of
the second kind satisfy%
\begin{align*}%
\genfrac{\{}{\}}{0pt}{}{i}{j}
&  =\frac{1}{j!}\sum_{k=j}^{i}\left(  -1\right)  ^{k-j}%
\genfrac{\{}{\}}{0pt}{}{i+1}{k+1}%
k!,\text{ for }j\geq0,\text{ }p=1\text{ in (\ref{7a2-1})}\\
&  =\frac{1}{i\left(  j-1\right)  !}\sum_{k=j}^{i}\left(  -1\right)
^{k-j}\frac{k!}{k-j+1}%
\genfrac{\{}{\}}{0pt}{}{i+1}{k+1}%
,\text{ for }j\geq1,\text{ }r=1\text{ in (\ref{7a2-2})}\\
&  =\frac{1}{\left(  j-i\right)  \left(  j-1\right)  !}\sum_{k=j+1}^{i}%
\frac{\left(  -1\right)  ^{k-j}k!}{\left(  k-j+1\right)  \left(  k-j\right)  }%
\genfrac{\{}{\}}{0pt}{}{i}{k}%
,\text{ for }j\geq1,\text{ }r=1\text{ in (\ref{7a2-3})}%
\end{align*}

The following are special cases of (\ref{6}) for $\lambda=0,$ nonnegative
integers $y=p,$ $x=r$ and $j\geq h:$%
\begin{align*}
\dbinom{i}{j}B_{i-j}(p)  &  =\sum\limits_{k=j}^{i}\left(  -1\right)
^{i-k}\frac{i}{k}%
\genfrac{\{}{\}}{0pt}{}{i}{k}%
\genfrac{[}{]}{0pt}{}{k+p}{j+p}%
_{p},\text{ for }h=1,\text{ }r=1,\\
B_{i}(r)  &  =\sum\limits_{k=0}^{i}\left(  -1\right)  ^{k}\frac{k!}{k+1}%
\genfrac{\{}{\}}{0pt}{}{i+r}{k+r}%
_{r},\text{ for }h=j=1,\text{ }p=0,\\
\dbinom{i+1}{j}  &  =\sum\limits_{k=j}^{i}\left(  -1\right)  ^{i-k}\left(
k+1\right)
\genfrac{\{}{\}}{0pt}{}{i+p}{k+p}%
_{p-1}%
\genfrac{[}{]}{0pt}{}{k+p}{j+p}%
_{p},\text{ for }h=-1,\text{ }r=p-1,\\
\dbinom{i+1}{j}  &  =\sum\limits_{k=j}^{i}\left(  -1\right)  ^{k-j}\left(
k+1\right)
\genfrac{\{}{\}}{0pt}{}{i+1+p}{k+1+p}%
_{p}%
\genfrac{[}{]}{0pt}{}{k+p}{j+p}%
_{p},\text{ for }h=-1,\text{ }r=p.
\end{align*}
Moreover, from (\ref{6}) for $\lambda=y=0,$ $h=j=1$ and $x=1-i,$ we have%
\[
B_{i}\left(  i+1\right)  =\sum\limits_{k=0}^{i}\frac{\left(  i-k\right)
!}{i-k+1}%
\genfrac{\{}{\}}{0pt}{}{i+k}{i}%
_{k}%
\]
by making use of $R_{2}\left(  m,k,-x\right)  =\left(  -1\right)  ^{m-k}%
R_{2}\left(  m,k,x-k\right)  $ and $B_{m}\left(  1-x\right)  =\left(
-1\right)  ^{m}B_{m}\left(  x\right)  .$

\subsection{Hyperharmonic numbers}

The hyperharmonic number of order $r$ denoted by $H_{m}^{r}$ is defined by%
\[
H_{m}^{r}=\sum\limits_{k=1}^{m}H_{k}^{r-1}%
\]
for \ $r,m\geq1,$ $H_{m}^{0}=\frac{1}{m}$ for $m\geq1,$ and $H_{m}^{r}=0$ for
$r<0$ or $m\leq0$ (\cite{26}).

A generating function for the hyperharmonic numbers is%
\[
-\left(  1-x\right)  ^{-r}\ln\left(  1-x\right)  =\sum\limits_{n=1}^{\infty
}H_{n}^{r}x^{n}.
\]
It follows from (\ref{stx1}) and (\ref{11}) that
\[
m!H_{m}^{r}=%
\genfrac{[}{]}{0pt}{}{m+r}{1+r}%
_{r}=S_{1}\left(  m,1|-1,0,r\right)  =\left(  -1\right)  ^{m-1}S_{1}\left(
m,1|1,0,-r\right)  .
\]
A combinatorial proof of this fact can be found in \cite[Theorem 2]{26}. Thus,
we have from (\ref{71}) \
\begin{align*}
i!H_{i+1}^{r}  &  =\sum_{k=0}^{i}\left(  -1\right)  ^{k}B_{k}%
\genfrac{[}{]}{0pt}{}{i+r}{k+r}%
_{r},\text{ for }h=1,\text{ }y=r,\text{ }x=j=\lambda=0,\\
i!H_{i+1}^{r-1}  &  =\sum_{k=0}^{i}B_{k}%
\genfrac{[}{]}{0pt}{}{i+r}{k+r}%
_{r},\text{ for }h=x=1,\text{ }y=r,\text{ }j=\lambda=0,\\
i!H_{i}^{r+p}  &  =\sum_{k=1}^{i}k%
\genfrac{[}{]}{0pt}{}{i+r}{k+r}%
_{r}p^{k-1},\text{ for }j=1,\text{ }y=r,\text{ }y-x=r+p\geq0,\text{ }%
h=\lambda=0,\\
i!H_{i}^{p}  &  =\sum_{k=1}^{i}k%
\genfrac{[}{]}{0pt}{}{i}{k}%
p^{k-1}%
\end{align*}
and from (\ref{5a1})%
\[
H_{i}^{r-m}=\sum_{k=1}^{i}\frac{\left(  -1\right)  ^{i-k}}{\left(  i-k\right)
!}\left(  m\right)  _{i-k}H_{k}^{r}%
\]
for $j=1,y=r,$ $y-x=r-m\geq0$ and $h=\lambda=0.$ This gives
\begin{align}
&  \sum_{k=1}^{i}\dbinom{i-k+p-1}{p-1}H_{k}^{r}=H_{i}^{p+r},\text{ for
}m=-p<0,\label{13}\\
&  \sum\limits_{k=1}^{i}kH_{k}^{r}=\left(  i+1\right)  !H_{i}^{r+1}%
-i!H_{i}^{r+2},\text{ for }p=2.\nonumber
\end{align}
(\ref{13}) is Eq. (7) of \cite{26} and thereby \cite[Theorem 1]{26} (see also
\cite[Theorem 5]{25a}) for $r=0$.

\subsection{Lah numbers}

For $\lambda=1$ and $x=0$ in (\ref{11}) and (\ref{11a}), we have
\[
\hspace{-0.3in}S_{1}\left(  m,k|1,-1,0\right)  =\left(  -1\right)
^{m-k}L\left(  m,k\right)  \text{ and }S_{2}\left(  m,k|1,-1,0\right)
=L\left(  m,k\right)  .
\]
Then, it follows from (\ref{6a}), (\ref{7a1}) and (\ref{7a2}) that%
\begin{align*}
&  \dbinom{i-j+m-1}{i-j}=\sum_{k=j}^{i}(-1)^{k-j}\binom{i+m-1}{k+m-1}%
\binom{k-1}{j-1},\text{ for }x=y=0,\text{ }h=-m<0,\\
&  \binom{i+h}{j+h}=\sum_{k=j}^{i}\binom{i}{k}\binom{h}{k-j},\text{ for
}x=y=0,\text{ }h\geq0,\\
&  \binom{i-h+1}{j-h+1}=\sum_{k=j}^{i}(-1)^{k-j}\binom{i+1}{k+1}%
\binom{k-j+h-1}{h-1},\text{ for }x=0,\text{ }y=1-\lambda,\text{ }j\geq h\geq1
\end{align*}
and $\lambda=-1,$ respectively.

In general, for arbitrary $x$ and $y$, we have
\begin{align*}
S_{2}\left(  m,j|1,-1,-x\right)   &  =(-1)^{m-j}S_{1}\left(
m,j|1,-1,-x\right)  =\binom{m}{j}\left\langle x+j\right\rangle _{m-j},\\
\beta_{m}^{\left(  h\right)  }(-1,x)  &  =\left(  -1\right)  ^{m}\alpha
_{m}^{\left(  h\right)  }(-1,-x)=\left\langle x-h\right\rangle _{m}.
\end{align*}
Then (\ref{6a}), (\ref{51}) and (\ref{7a2}) reduce to\textbf{\ }%
\begin{align*}
\left\langle x-h-y\right\rangle _{i-j}  &  =\sum_{k=0}^{i-j}\binom{i-j}%
{k}(-1)^{k}\left\langle y+j\right\rangle _{k}\left\langle x-h+j+k\right\rangle
_{i-j-k},\\
\left\langle y+j+x-h\right\rangle _{i-j}  &  =\sum_{k=0}^{i-j}\binom{i-j}%
{k}\left\langle y+j\right\rangle _{k}\left\langle x-h\right\rangle _{i-j-k},\\
\left\langle y+j-x-h\right\rangle _{i-j}  &  =\sum_{k=0}^{i-j}\binom{i-j}%
{k}(-1)^{k}\left\langle x+h\right\rangle _{k}\left\langle y+j+k\right\rangle
_{i-j-k},
\end{align*}
respectively.

\end{document}